\documentclass{birkjour}
 \newtheorem{thm}{Theorem}[section]
 \newtheorem{cor}[thm]{Corollary}
 \newtheorem{lem}[thm]{Lemma}
\newtheorem{question}[thm]{Question}
 \newtheorem{prop}[thm]{Proposition}
\newtheorem{prob}[thm]{Problem}
 \theoremstyle{definition}
 \newtheorem{defn}[thm]{Definition}
 \theoremstyle{remark}
 
\newtheorem{ex}[thm]{Example}
 \numberwithin{equation}{section}

\newcommand{\mC}{\ensuremath{\mathbb{C}}}
\newcommand{\mD}{\ensuremath{\mathbb{D}}}

\newcommand{\mN}{\ensuremath{\mathbb{N}}}

\newcommand{\cM}{\ensuremath{\mathcal{M}}}
\newcommand{\cF}{\ensuremath{\mathcal{F}}}

\newcommand{\cH}{\ensuremath{\mathcal{H}}}

\begin{document}

%
%
%
%
%
%
%
%
%

\title[A Normal Criterion Concerning Sequence of Functions...]
 {A Normal Criterion Concerning Sequence of Functions and their Differential Polynomials}

\author[N. Bharti]{Nikhil Bharti}

\address{%
Department of Mathematics\\
University of Jammu\\
Jammu-180006, India}
\email{nikhilbharti94@gmail.com}





\begin{abstract}
In this paper, a normality criterion concerning a sequence of meromorphic functions and their differential polynomials is obtained. Precisely, we have proved: Let $\left\{f_j\right\}$ be a sequence of meromorphic functions in the open unit disk $\mD$ such that, for each $j,$ $f_j$ has poles of multiplicity at least $m,~m\in\mN.$ Let $\left\{h_j\right\}$ be a sequence of holomorphic functions in $\mD$ such that $h_j\rightarrow h$ locally uniformly in $\mD,$ where $h$ is holomorphic in $\mD$ and $h\not\equiv 0.$ Let $Q[f_j]$ be a differential polynomial of $f_j$ having degree $\lambda_Q$ and weight $\mu_Q.$ If, for each $j,$ $f_j(z)\neq 0$ and $Q[f_j]-h_j$ has at most $\mu_Q + \lambda_Q(m-1)-1$ zeros, ignoring multiplicities, in $\mD,$ then $\left\{f_j\right\}$ is normal in $\mD.$ 
\end{abstract}

\maketitle

\section{Introduction}
In what follows, $\mathcal{H}(D)$ and $\mathcal{M}(D)$ are the classes of all holomorphic and meromorphic functions in the domain $D\subseteq\mC,$ respectively. A family $\mathcal{F}\subset\cM(D)$ is said to be {\it normal} in $D$ if every sequence of functions in $\mathcal{F}$ has a subsequence which converges locally uniformly in $D$ with respect to the spherical metric to a limit function which is either meromorphic in $D$ or the constant $\infty.$ In case, if $\mathcal{F}\subset\cH(D),$ then the Euclidean metric can be substituted for the spherical metric (see \cite{schiff, zalcman}). The idea of normal family is attributed to Paul Montel \cite{montel-1, montel-2, montel-3}. Ever since its creation, the theory of normal families has been a cornerstone of complex analysis with far reaching applications in dynamics of rational as well as transcendental maps, function theory of one and several variables, bicomplex analysis, complex projective geometry, functional analysis etc. (see \cite{aladro, beardon-minda, KRN-1, KRN-2, fujimoto, kim, wu, zalcman}). The theory of normal families continues to be an active area of research. 

The main purpose of this paper is to study the normality of a sequence of non-vanishing meromorphic functions in a domain $D\subseteq\mC$ whose differential polynomials have non-exceptional holomorphic functions in $D.$ For $f,~g\in\cM(D),$ if $f(z)-g(z)\neq 0$ in $D,$ then $g$ is said to be an exceptional function of $f$ in $D.$ On the other hand, if there exist at least one $z\in D$ for which $f(z)-g(z)=0,$ then $g$ is said to be a non-exceptional function of $f$ in $D.$ If $g$ happens to be a constant, say $k,$ then $k$ is said to be an exceptional (respectively, non-exceptional) value of $f$ in $D.$

\medskip

\begin{defn}\cite{grahl}
Let $k\in\mN$ and $n_0, n_1,\ldots, n_k$ be non-negative integers,  not all zero. Let $f\in\mathcal{M}(D).$ Then the product $$M[f]:=a\cdot\prod\limits_{j=0}^{k}\left(f^{(j)}\right)^{n_j}$$ is called a differential monomial of $f,$ where $a~(\not\equiv 0, \infty)\in\mathcal{M}(D).$ If $a\equiv 1,$ then $M[f]$ is said to be a normalized differential monomial of $f.$ The quantities $$\lambda_M:=\sum\limits_{j=0}^{k}n_j \mbox{ and } \mu_M:=\sum\limits_{j=0}^{k}(j+1)n_j$$ are called the degree and weight of the differential monomial $M[f],$ respectively. 

For $1\leq i\leq m,$ let $M_i[f]=\prod\limits_{j=0}^{k}\left(f^{(j)}\right)^{n_{ji}}$  be $m$ differential monomials of $f.$ Then the sum $$P[f]:= \sum\limits_{i=1}^{m}a_iM_i[f]$$ is called a differential polynomial of $f$ and the quantities $$\lambda_P:=\max\left\{\lambda_{M_i}: 1\leq i\leq m\right\} \mbox{ and } \mu_P:=\max\left\{\mu_{M_i}: 1\leq i\leq m\right\}$$ are called the degree and weight of the differential polynomial $P[f],$ respectively. If $\lambda_{M_1}=\lambda_{M_2}=\cdots=\lambda_{M_m},$ then $P[f]$ is said to be a homogeneous differential polynomial.
\end{defn}

\medskip

In the present work, we are concerned with the homogeneous differential polynomials of the form 

\begin{equation}\label{eqn:3}
Q[f]:=f^{x_0}(f^{x_1})^{(y_1)}(f^{x_2})^{(y_2)}\cdots(f^{x_k})^{(y_k)},
\end{equation}
where $x_0,~x_1,~\ldots,~x_k,~y_1,~y_2,~\ldots,~y_k$ are non-negative integers such that $x_i\geq y_i$ for $i=1,~2,~\ldots,~k.$\\
The differential polynomial \eqref{eqn:3} first appeared in the literature in \cite{dethloff} and has been used extensively since then, particularly in finding normality criteria of families of meromorphic functions (see \cite{thin-1, thin-2, thin-3}).\\ 
Set $x'=\sum\limits_{i=1}^{k} x_i$ and  $y'=\sum\limits_{i=1}^{k}y_i.$ Further, we assume that $x_0>0$ and $y'>0.$ Using the generalized Leibniz rule for derivatives, one can easily verify that $$(f^{x_i})^{(y_i)}=\sum\limits_{n_1+n_2+\cdots+n_{x_i}=y_i} \frac{y_i!}{n_1!n_2!\cdots n_{x_i}!}f^{(n_1)}f^{(n_2)}\cdots f^{(n_{x_i})},$$ where $n_i$'s are non-negative integers. Thus the degree of $Q[f],$ $\lambda_Q=x_0+x'$ and the weight of $Q[f],$ $\mu_Q=x_0+x'+y'=\lambda_Q+y'.$

\section{A Glimpse of the Past}
In $1907,$ Montel \cite{montel-1} proved that a locally uniformly bounded sequence of holomorphic functions in a domain has a locally uniformly convergent subsequence. Montel called a family of holomorphic functions exhibiting such behaviour, a {\it normal family} in his subsequent paper \cite{montel-2} in $1912.$ However, the fundamental ideas related to the study of normal convergence were developed long before Montel. In $1860$s, Weierstrass established that a uniform limit of a sequence of holomorphic functions in a domain is also holomorphic (see \cite{remmert}). In $1894,$ Stieltjes \cite{stieltjes} proved that a uniformly bounded sequence of holomorphic functions in a domain converging uniformly in some open subset of the domain, also converges uniformly on compacta. These results were further refined by Osgood \cite{osgood}, Vitali \cite{vitali-1, vitali-2} and Porter \cite{porter-1, porter-2} in early $1900$s. Montel's definition of normality of a family of holomorphic functions was later extended by Carath$\acute{e}$odary and Landau \cite{caratheodory-landau} to a family of meromorphic functions by means of local uniform convergence in the spherical metric. As a matter of fact, it was pointed out by Ostrowski \cite{ostrowski} that normality can be equivalently defined by using the spherical metric. 
In $1967,$ Wu \cite{wu} studied normal families of holomorphic mappings between complex manifolds. Later, Fujimoto \cite{fujimoto} introduced the concept of meromorphically normal families of meromorphic mappings in complex projective spaces. In $1969,$ Drasin \cite{drasin} introduced the value distribution theory of Nevanlinna in the study of normal families. For a detailed account on the history of the theory of normal families, the reader is referred to Zalcman \cite{zalcman} and Beardon and Minda \cite{beardon-minda}.\\
A major breakthrough in the theory of normal families happened in $1975$ when Zalcman \cite{zalcman-1} proved his famous rescaling lemma, known as the Zalcman's lemma, which unlike normal families, gave a characterization of non-normal families in a domain. This lemma provided an easy approach to obtain normality of a family of functions using {\it reductio ad absurdum} and has revolutionized the theory of normal families. Several extensions of Zalcman's lemma in one and several variables have been obtained (see Pang and Zalcman \cite{pang-zalcman}, Aladro and Krantz \cite{aladro}, Charak and Sharma \cite{charak-1} and Dovbush \cite{dovbush} ). Kim and Krantz \cite{kim} considered normal families of holomorphic functions and mappings on an infinite dimensional Banach space. Charak et al. \cite{KRN-1, KRN-2} introduced normal families of bicomplex holomorphic and meromorphic functions. Recently, Krantz \cite{krantz} investigated normal families of holomorphic mappings from the point of view of invariant geometry and Chang et al. \cite{chang-xu-yang} studied normal families of meromorphic functions in several complex variables.

\section{Motivation and Main Result}
In \cite[Problem 5.11]{hayman-1}, Hayman posed the following problem:

\begin{prob}\label{prob:1}
Let $\cF\subset\cM(D)$ and $k$ be a positive integer. Suppose that for each $f\in\cF,$ $f(z)\neq 0,$ $f^{(k)}\neq 1.$ Then what can be said about the normality of $\cF$ in $D?$
\end{prob}

Gu \cite{gu} considered Problem \ref{prob:1} and confirmed that the family $\cF$ is indeed normal in $D.$ Subsequently, Yang \cite{yang} proved that the exceptional value $1$ of $f^{(k)}$ can be replaced by an exceptional holomorphic function. Chang \cite{chang} considered the case when $f^{(k)}-1$ has limited number of zeros and obtained the normality of $\cF.$ Thin and Oanh \cite{thin-3} replaced $f^{(k)}$ with a differential polynomial of $f.$ Later, Deng et al. \cite{deng} established that there is no loss of normality even when $f^{(k)}-h$ has zeros for some $h\in\cH(D)$ as long as the number of zeros are bounded by the constant $k.$ Chen et al. \cite{chen} took a sequence of exceptional holomorphic functions instead of single exceptional holomorphic function. Recently, Deng et al. \cite{deng-1} proved the following theorem concerning a sequence of meromorphic functions:

\begin{thm}\label{thm:deng}
Let $\left\{f_j\right\}\subset\mathcal{M}(D)$ and $\left\{h_j\right\}\subset\cH(D)$ be sequences of functions in $D.$ Assume that $h_j\rightarrow h$ locally uniformly in $D,$ where $h\in\cH(D)$ and $h\not\equiv 0.$ Let $k$ be a positive integer. If, for each $j,$ $f_j(z)\neq 0$ and  $f_{j}^{(k)}-h_j(z)$ has at most $k$ distinct zeros, ignoring multiplicities, in $D,$ then $\left\{f_j\right\}$ is normal in $D.$
\end{thm}

Following Thin and Oanh \cite{thin-3}, a natural question about Theorem \ref{thm:deng} arises:

\begin{question}\label{que:1}
Does Theorem \ref{thm:deng} remain valid if the $k$-th derivative $f_{j}^{(k)}$ is replaced by a differential polynomial of $f_j?$
\end{question} 
We give an affirmative answer to Question \ref{que:1} in the following result. Since normality is a local property, one can always restrict the domain to the open unit disk $\mD.$

\begin{thm}\label{thm:1}
Let $\left\{f_j\right\}\subset\mathcal{M}(\mD),$ be a sequence such that, for each $j,$ $f_j$ has poles of multiplicity at least $m,~m\in\mN.$ Let $\left\{h_j\right\}\subset\mathcal{H}(\mD)$ be such that $h_j\rightarrow h$ locally uniformly in $\mD,$ where $h\in\cH(\mD)$ and $h\not\equiv 0.$ Let $Q[f_j]$ be a differential polynomial of $f_j$ as defined in \eqref{eqn:3} having degree $\lambda_Q$ and weight $\mu_Q.$ If, for each $j,$ $f_j(z)\neq 0$ and $Q[f_j]-h_j$ has at most $\mu_Q + \lambda_Q(m-1)-1$ zeros, ignoring multiplicities, in $\mD,$ then $\left\{f_j\right\}$ is normal in $\mD.$
\end{thm}

It is important to note that the multiplicity of poles of $f_j$ in Theorem \ref{thm:1} is considered only to relax the upper bound for the zeros of $Q[f_j]-h_j$ in $\mD.$ In fact, Theorem \ref{thm:1} can be proved without taking the multiplicity of poles of $f_j$ into account. In such case, the hypothesis that $Q[f_j]-h_j$ has at most $\mu_Q + \lambda_Q(m-1)-1$ zeros, ignoring multiplicities, in $\mD,$ is replaced by the hypothesis that $Q[f_j]-h_j$ has at most $\mu_Q-1$ zeros, ignoring multiplicities, in $\mD.$

As a direct consequence of Theorem \ref{thm:1}, we have

\begin{cor}
Let $\left\{f_j\right\}\subset\mathcal{M}(\mD)$ and $\left\{h_j\right\}\subset\mathcal{H}(\mD)$ be such that $h_j\rightarrow h$ locally uniformly in $\mD,$ where $h\in\cH(D)$ and $h\not\equiv 0.$ If, for each $j,$ $f_j(z)\neq 0$ and $Q[f_j](z)\neq h_j(z),$ then $\left\{f_j\right\}$ is normal in $\mD.$
\end{cor}

In the following, we show that the condition `$f_j(z)\neq 0$' in Theorem \ref{thm:1} is essential.

\begin{ex}
Consider a sequence $\left\{f_j\right\}\subset\cM(\mD)$ given by $f_j(z)=jz,~j\in\mN,~j\geq 2.$ Let $Q[f_j]:=f_jf_j'$ so that $\mu_Q=3,$ and let $h_j(z)=z.$ Then $h_j\rightarrow z\not\equiv 0$ and $Q[f_j](z)-h_j(z)$ has at most one zero in $\mD.$ However, $\left\{f_j\right\}$ is normal in $\mD.$
\end{ex}

Taking $h_j(z)=1/z$ in Example \ref{exp:d3}, we find that $h_j$ cannot be meromorphic in $\mD.$ Furthermore, the condition `$h\not\equiv 0$' in Theorem \ref{thm:1} cannot be dropped as demonstrated by the following example:

\begin{ex}
Let $\left\{f_j\right\}\subset\cM(\mD)$ such that $f_j(z)=e^{jz},~j\in\mN$ and let $h_j\equiv 0$ so that $h_j\rightarrow h\equiv 0.$ Let $Q[f_j]$ be any differential polynomial of $f_j$ of the form \eqref{eqn:3}. Clearly, $Q[f_j](z)-h(z)$ has no zero in $\mD.$ But the sequence $\left\{f_j\right\}$ is not normal in $\mD.$  
\end{ex}	

The following example establishes the sharpness of the condition `$Q[f_j]-h_j$ has at most $\mu_Q + \lambda_Q(m-1)-1$ distinct zeros in $\mD$' in Theorem \ref{thm:1}:

\begin{ex}\label{exp:d3}
Let $\left\{f_j\right\}\subset\cM(\mD)$ such that $$f_j(z)=\frac{1}{jz},~j\geq 3,~j\in\mN$$ and let $Q[f_j]:=f_jf_j'.$ Then $\lambda_Q=2,$ $\mu_Q=3,~m=1$ and $Q[f_j](z)=-1/j^2z^3.$ Consider $h_j(z)=1/(z-1)^3$ so that $\left\{h_j\right\}\in\mathcal{H}(\mD)$ and $h_j\rightarrow 1/(z-1)^3\not\equiv 0.$ Then by simple calculations, one can easily see that $Q[f_j](z)-h_j(z)$ has exactly $\mu_Q + \lambda_Q(m-1)=3$ distinct zeros in $\mD.$ However, the sequence $\left\{f_j\right\}$ is not normal in $\mD.$
\end{ex}

\section{Auxiliary Results}
What follow are the preparations for the proof of the main result. Here, we assume that the reader is familiar with standard notations of Nevanlinna's value distribution theory like $m(r,f),~N(r,f),~T(r,f),~S(r,f)$ (see \cite{hayman, lo}). Recall that a function $g\in\cM(\mC)$ is said to be a small function of $f\in\cM(\mC)$ if $T(r,g)=S(r,f)$ as $r\rightarrow\infty,$ possibly outside a set of finite Lebesgue measure.

\medskip

\begin{flushleft}
{\bf Notation:} By $D_r(a),$ we mean an open disk in $\mC$ with center $a$ and radius $r.$ $\mD=D_1(0)$ is the open unit disk in $\mC.$
\end{flushleft}

\medskip 

The following lemma is an extension of the Zalcman-Pang Lemma due to Chen and Gu \cite{chen-gu} (cf. \cite[Lemma 2]{pang-zalcman}). 

\begin{lem}[Zalcman-Pang Lemma]\label{lem:zp}
Let $\mathcal{F}\subset\mathcal{M}(\mD)$ be such that each $f\in\cF,$ has zeros of multiplicity at least $m$ and poles of multiplicity at least $p.$ Let $-p<\alpha<m.$ If $\mathcal{F}$ is not normal at $z_0\in\mD,$ then there exist
sequences $\left\{f_j\right\}\subset\mathcal{F},$ $\left\{z_j\right\}\subset\mD$ satisfying $z_j\rightarrow z_0$ and positive numbers $\rho_j$ with $\rho_j\rightarrow 0$ such that the sequence $\left\{g_j\right\}$ defined by $$g_j(\zeta)=\rho_{j}^{-\alpha}f_j(z_j+\rho_j\zeta)\rightarrow g(\zeta)$$ locally uniformly in $\mathbb{C}$ with respect to the spherical metric, where $g$ is a non-constant meromorphic function on $\mathbb{C}$ such that for every $\zeta\in\mathbb{C},$ $g^{\#}(\zeta)\leq g^{\#}(0)=1.$
 \end{lem}

We remark that if $f(z)\neq 0$ in $D$ for every $f\in\cF,$ then $\alpha\in (-p,~ +\infty).$ Likewise, if each $f\in\cF$ does not have any pole in $D,$ then $\alpha\in (-\infty,~m)$ and if $f(z)\neq 0,~\infty$ in $D,$ for every $f\in\cF,$ then $\alpha\in (-\infty,~+\infty).$ 
    
\begin{lem}\cite[Lemma 3]{chen-xu}\label{lem:xu}
Let $\cF\subset\cM(\mD)$ and suppose that $h\in\cH(\mD)$ or $h\equiv\infty.$ Further, assume that, for each $f\in\cF,$ $f(z)\neq h(z)$ in $\mD.$ If $\cF$ is normal in $\mD\setminus\left\{0\right\}$ but not normal in $\mD,$ then there exists a sequence $\left\{f_j\right\}\subset\cF$ such  that $f_j\rightarrow h$ in $\mD\setminus\left\{0\right\}.$
\end{lem}		
		
\begin{prop}\label{prop:1}
Let $f\in\mathcal{M}(\mC)$ be a transcendental function and let $Q[f]$ be a differential polynomial of $f$ as defined in \eqref{eqn:3} having degree $\lambda_Q$ and weight $\mu_Q.$ Assume that $\psi~(\not\equiv 0,~\infty)$ is a small function of $f.$ Then $$\lambda_Q T(r,f)\leq\overline{N}(r,f) + (1+\mu_Q-\lambda_Q)\overline{N}\left(r,\frac{1}{f}\right) + \overline{N}\left(r,\frac{1}{Q[f]-\psi}\right)+ S(r,f).$$ 
\end{prop}

\begin{proof}
By definition of $Q[f],$ it is apparent that $Q[f]\not\equiv 0.$ Then from the first fundamental theorem of Nevanlinna, we have
\begin{align}\label{eq:m2}
	\lambda_QT(r,f) &= \lambda_Q~m\left(r,\frac{1}{f}\right)+\lambda_QN\left(r,\frac{1}{f}\right)+O(1)\nonumber\\
	&\leq m\left(r,\frac{Q[f]}{f^{\lambda_Q}}\right)+ m\left(r,\frac{1}{Q[f]}\right) +\lambda_QN\left(r,\frac{1}{f}\right)+O(1).
	\end{align}
From the Nevanlinna's theorem on logarithmic derivative, we find that $$m\left(r,\frac{Q[f]}{f^{\lambda_Q}}\right)=S(r,f).$$
Thus from \eqref{eq:m2}, we obtain
\begin{align*}
	\lambda_QT(r,f) &\leq m\left(r,\frac{1}{Q[f]}\right) +\lambda_QN\left(r,\frac{1}{f}\right)+ S(r,f)\\
	&= T(r,Q[f])- N\left(r,\frac{1}{Q[f]}\right)+ \lambda_QN\left(r,\frac{1}{f}\right)+S(r,f).
\end{align*}
Applying the second fundamental theorem of Nevanlinna for small functions to $T(r,Q[f]),$ we get
\begin{align}\label{eq:l3_2}
	\lambda_QT(r,f) &\leq \lambda_QN\left(r,\frac{1}{f}\right)+\overline{N}(r,Q[f])+\overline{N}\left(r,\frac{1}{Q[f]}\right)+\overline{N}\left(r,\frac{1}{Q[f]-\psi}\right)\nonumber\\ &\qquad-N\left(r,\frac{1}{Q[f]}\right)+ S(r,f)\nonumber\\
	&= \lambda_QN\left(r,\frac{1}{f}\right)+\overline{N}(r,f)+ \overline{N}\left(r,\frac{1}{Q[f]}\right)+\overline{N}\left(r,\frac{1}{Q[f]-\psi}\right)\\ \nonumber &\qquad-N\left(r,\frac{1}{Q[f]}\right)+ S(r,f).
\end{align} 

Since a zero of $f$ with multiplicity $m$ is also a zero of $Q[f]$ with multiplicity at least $(m+1)\lambda_Q-\mu_Q,$ $$N\left(r,\frac{1}{Q[f]}\right)- \overline{N}\left(r,\frac{1}{Q[f]}\right)\geq\left[(m+1)\lambda_Q-\mu_Q-1\right]\overline{N}\left(r,\frac{1}{f}\right).$$
Therefore, from \eqref{eq:l3_2}, we obtain
\begin{align*}
	\lambda_QT(r,f) &\leq \lambda_QN\left(r,\frac{1}{f}\right)+\overline{N}(r,f)+\left[1+\mu_Q-(m+1)\lambda_Q\right]\overline{N}\left(r,\frac{1}{f}\right)\\ &\qquad + \overline{N}\left(r,\frac{1}{Q[f]-\psi}\right)+ S(r,f)\\
	&\leq  \overline{N}(r,f) + (1+\mu_Q-\lambda_Q)\overline{N}\left(r,\frac{1}{f}\right) + \overline{N}\left(r,\frac{1}{Q[f]-\psi}\right)+ S(r,f).
	\end{align*}
\end{proof}		
		
	\begin{cor}\label{cor:1}
		Let $f\in\mathcal{M}(\mC)$ be a transcendental function and let $Q[f]$ be a differential polynomial of $f$ as defined in \eqref{eqn:3}. Assume that $\psi~(\not\equiv 0,~\infty)$ is a small function of $f.$ If $f\neq 0,$ then $Q[f]-\psi$ has infinitely many zeros in $\mC.$
	\end{cor}	
		
		\begin{proof}
		From Proposition \ref{prop:1}, we have
		\begin{align}\label{eq:cor1}
		\lambda_Q T(r,f)& \leq\overline{N}(r,f) + (1+\mu_Q-\lambda_Q)\overline{N}\left(r,\frac{1}{f}\right) + \overline{N}\left(r,\frac{1}{Q[f]-\psi}\right)\nonumber\\ &\qquad + S(r,f).
		\end{align}
		Since $f\neq 0,$ $\overline{N}(r,1/f)=0.$ Thus from \eqref{eq:cor1}, we obtain
		\begin{align*}
		\lambda_Q T(r,f)& \leq\overline{N}(r,f) + \overline{N}\left(r,\frac{1}{Q[f]-\psi}\right)+ S(r,f).
		\end{align*}
		This implies that $$(\lambda_Q-1)T(r,F)\leq\overline{N}\left(r,\frac{1}{Q[F]-\psi}\right)+ S(r,F).$$ Since $\lambda_Q-1>0,$ it follows that $Q[F]-\psi$ has infinitely many zeros in $\mC.$
		\end{proof}
		
In \cite{chang}, Chang proved that if $f$ is a non-constant rational function such that $f\neq 0,$ then for $k\geq 1,$ $f^{(k)}-1$ has at least $k+1$ distinct zeros in $\mC.$ Using the method of Chang \cite{chang}, Deng et al. \cite{deng} proved that the constant $1$ can be replaced by a polynomial $p~(\not\equiv 0).$ Recently, Xie and Deng \cite{xie} sharpened the lower bound for the distinct zeros of $f^{(k)}-p$ in $\mC$ by involving the multiplicity of poles of $f.$ Thin and Oanh \cite{thin-3} extended the result of Chang to differential polynomials by proving that if $f~(\neq 0)$ is a non-constant rational function, then $Q[f]-1$ has at least $\mu_Q$ distinct zeros in $\mC.$ We obtain a slightly better result in the following form:
 
\begin{prop}\label{prop:2}
	Let $f$ be a non-constant rational function, having poles of multiplicity at least $m,$ $m\in\mN,$ and let $p~(\not\equiv 0)$ be a polynomial. Let $Q[f]$ be a differential polynomial of $f$ as defined in \eqref{eqn:3} having degree $\lambda_Q$ and weight $\mu_Q.$ Assume that $f\neq 0.$ Then $Q[f]-p$ has at least $\mu_Q+\lambda_Q(m-1)$ distinct zeros in $\mathbb{C}.$  
\end{prop}	

\begin{proof}
Since $f\neq 0,$ it follows that $f$ cannot be a polynomial and so $f$ has at least one pole. Therefore, we can write
	\begin{equation}\label{eq:l1.1}
		f(z)=\frac{C_1}{\prod\limits_{i=1}^{n}(z+\alpha_i)^{n_i+m-1}}.
	\end{equation}
Let
	\begin{equation}\label{eq:l1.2}
		p(z)=C_2\prod\limits_{i=1}^{l}(z+ \beta_i)^{l_i},
	\end{equation}
	where $C_1,~C_2$ are non-zero constants; $l,~n,~n_i$ are positive integers and $l_i$ are non-negative integers. Also, $\beta_i$ (when $1\leq i\leq l$) are distinct complex numbers and $\alpha_i$ (when $1\leq i\leq n$) are distinct complex numbers.
	
	From \eqref{eq:l1.1}, one can deduce that 
	\begin{equation}\label{eq:l1.3}
		Q[f](z)=\frac{h_Q(z)}{\prod\limits_{i=1}^{n}(z+\alpha_i)^{\lambda_Q(n_i+m-1)+\mu_Q-\lambda_Q}},
	\end{equation}
	where  $h_Q$ is a polynomial of degree $(n-1)(\mu_Q-\lambda_Q).$
	
	Also, it is easy to see that $Q[f]-p$ has at least one zero in $\mathbb{C}.$ Therefore, we can set
	\begin{equation}\label{eq:l1.4}
		Q[f](z)=p(z)+\frac{C_3\prod\limits_{i=1}^{q}(z+\gamma_i)^{q_i}}{\prod_{i=1}^{n}(z+\alpha_i)^{\lambda_Q(n_i+m-1)+\mu_Q-\lambda_Q}},
	\end{equation}
	where $C_3\in\mathbb{C}\setminus\left\{0\right\},$ $q_i$ are positive integers and $\gamma_i$ ($1\leq i\leq q$) are distinct complex numbers.
	
	Let $L=\sum\limits_{i=1}^{l}l_i$ and $N=\sum\limits_{i=1}^{n}n_i.$ Then from \eqref{eq:l1.2}, \eqref{eq:l1.3} and \eqref{eq:l1.4}, we have
	\begin{equation}\label{eq:l1.5}
		C_2\prod\limits_{i=1}^{l}(z+\beta_i)^{l_i}\prod_{i=1}^{n}(z+\alpha_i)^{\lambda_Q(n_i+m-1)+\mu_Q-\lambda_Q} + C_3\prod\limits_{i=1}^{q}(z+\gamma_i)^{q_i}=h_Q(z)
	\end{equation}
	From \eqref{eq:l1.5}, we find that 
	\begin{align*}
	\sum\limits_{i=1}^{q}q_i&=\sum\limits_{i=1}^{n}\left[\lambda_Q(n_i+m-1)+\mu_Q-\lambda_Q\right] + \sum\limits_{i=1}^{l}l_i\\
	&=\lambda_QN + n(m-1)\lambda_Q+n(\mu_Q-\lambda_Q)+ L
	\end{align*}
	and $C_3=-C_2.$
	
	Also, from \eqref{eq:l1.5}, we get
	\begin{align*}
	&\prod\limits_{i=1}^{l}(1+\beta_ir)^{l_i}\prod_{i=1}^{n}(1+\alpha_ir)^{\lambda_Q(n_i+m-1)+\mu_Q-\lambda_Q} - \prod\limits_{i=1}^{q}(1+\gamma_ir)^{q_i}\\
	&=r^{\mu_Q+\lambda_Q(N+ n(m-1)-1)+L}b(r),
\end{align*}	
	where $b(r):=r^{(n-1)(\mu_Q-\lambda_Q)}h_Q(1/r)/C_2$ is a polynomial of degree is at most $(n-1)(\mu_Q-\lambda_Q).$ Furthermore, it follows that
	\begin{align}\label{eq:l1.6}
		&\frac{\prod\limits_{i=1}^{l}(1+\beta_ir)^{l_i}\prod\limits_{i=1}^{n}(1+\alpha_ir)^{\lambda_Q(n_i+m-1)+\mu_Q-\lambda_Q}}{\prod\limits_{i=1}^{q}(1+\gamma_ir)^{q_i}}\nonumber\\
		&=1+\frac{r^{\mu_Q+\lambda_Q(N+n(m-1)-1)+L}b(r)}{\prod\limits_{i=1}^{q}(1+\gamma_ir)^{q_i}}\nonumber\\ 
		&=1+ O\left(r^{\mu_Q+\lambda_Q(N+n(m-1)-1)+L}\right) \mbox{ as } r\rightarrow 0.
	\end{align}
 Taking logarithmic derivatives of both sides of \eqref{eq:l1.6}, we obtain
	\begin{align}\label{eq:l1.7}
		\sum\limits_{i=1}^{l}\frac{l_i\beta_i}{1+\beta_ir}&+ \sum\limits_{i=1}^{n}\frac{\left[\lambda_Q(n_i+m-1)+\mu_Q-\lambda_Q\right]\alpha_i}{1+\alpha_ir}-\sum\limits_{i=1}^{q}\frac{q_i\gamma_i}{1+\gamma_ir}\nonumber\\
		&=O\left(r^{\mu_Q+\lambda_Q(N+ n(m-1)-1)+L-1}\right) \mbox{ as } r\rightarrow 0.
	\end{align}
	
	Let $S_1=\left\{\beta_1,\beta_2,\ldots,\beta_l\right\}\cap\left\{\alpha_1,\alpha_2,\ldots,\alpha_n\right\}$ and $S_2=\left\{\beta_1,\beta_2,\ldots,\beta_l\right\}\cap\left\{\gamma_1,\gamma_2,\ldots,\gamma_q\right\}.$ Then we consider the following cases:\\
{\bf  Case 1:} $S_1=S_2=\emptyset.$\\
Let $\alpha_{n+i}=\beta_i$ when $1\leq i\leq l$ and $$N_i=\left\{\begin{array}{cc} \lambda_Q(n_i+m-1)+\mu_Q-\lambda_Q & \mbox{if}~ 1\leq i\leq n,\\ l_{i-n} & \mbox{if}~ n+1\leq i\leq n+l.\end{array}\right.$$ Then \eqref{eq:l1.7} can be written as 
	\begin{equation}\label{eq:l1.8}
		\sum\limits_{i=1}^{n+l}\frac{N_i\alpha_i}{1+\alpha_ir}-\sum\limits_{i=1}^{q}\frac{q_i\gamma_i}{1+\gamma_ir}=O\left(r^{\mu_Q+\lambda_Q(N+ n(m-1)-1)+L-1}\right)~\mbox{as $r\rightarrow 0.$}
	\end{equation}
	Comparing the coefficients of $r^j,~j=0,1,\ldots,\mu_Q+\lambda_Q(N+ n(m-1)-1)+L-2$ in \eqref{eq:l1.8}, we find that 
	\begin{equation}\label{eq:l1.9}
		\sum\limits_{i=1}^{n+l}N_i\alpha_{i}^j-\sum\limits_{i=1}^{q}q_i\gamma_i^j=0,~\mbox{for each}~ j=1,2,\ldots, \mu_Q+\lambda_Q(N+ n(m-1)-1)+L-1.
	\end{equation}
	Now, let $\alpha_{n+l+i}=\gamma_i$ for $1\leq i\leq q.$ Then from \eqref{eq:l1.9} and the fact that $\sum\limits_{i=1}^{n+l}N_i-\sum\limits_{i=1}^{q}q_i=0,$ we deduce that the system of equations
	\begin{equation}\label{eq:l1.10}
		\sum\limits_{i=1}^{n+l+q}\alpha_i^jx_i=0,~j=0,1,\ldots, \mu_Q+\lambda_Q(N+ n(m-1)-1)+L-1,
	\end{equation}
	has a non-zero solution $$\left(x_1,\ldots,x_{n+l},x_{n+l+1},\ldots,x_{n+l+q}\right)=\left(N_1,\ldots,N_{n+l},-q_1,\ldots,-q_q\right).$$
	This is possible only when the rank of the coefficient matrix of the system \eqref{eq:l1.10} is strictly less than $n+l+q.$
	
Hence $\mu_Q+\lambda_Q(N+ n(m-1)-1)+L< n+l+q.$ Since $N=\sum\limits_{i=1}^{n}n_i\geq n$ and $L=\sum\limits_{i=1}^{l}l_i\geq l,$ it follows that $q\geq \mu_Q+\lambda_Q(m-1).$ \\
{\bf  Case 2:} $S_1\neq\emptyset$ and $S_2=\emptyset.$\\
We may assume, without loss of generality, that $S_1=\left\{\beta_1,\beta_2,\ldots,\beta_{s_1}\right\}.$ Then $\beta_i=\alpha_i$ for $1\leq i\leq s_1.$ Take $s_3=l-s_1.$\\
{\bf  Subcase 2.1:} $s_3\geq 1.$\\
Let $\alpha_{n+i}=\beta_{s_1+i}$ for $1\leq i\leq s_3$ and if $s_1<n,$ then let $$N_i=\left\{\begin{array}{ccc} \lambda_Q(n_i+m-1)+\mu_Q-\lambda_Q + l_i & \mbox{if } 1\leq i\leq s_1,\\  \lambda_Q(n_i+m-1)+\mu_Q-\lambda_Q & \mbox{if } s_1+1\leq i\leq n,\\  l_{s_1-n+i} & \mbox{if } n+1\leq i\leq n+s_3.\end{array}\right.$$ If $s_1=n,$ then we take $$N_i=\left\{\begin{array}{cc} \lambda_Q(n_i+m-1)+\mu_Q-\lambda_Q +l_i & \mbox{if } 1\leq i\leq s_1,\\ l_{s_1-n+i} & \mbox{if } n+1\leq i\leq n+s_3.\end{array}\right.$$
{\bf  Subcase 2.2:} $s_3=0.$\\
 If $s_1<n,$ then set $$N_i=\left\{\begin{array}{cc} \lambda_Q(n_i+m-1)+\mu_Q-\lambda_Q+l_i & \mbox{if } 1\leq i\leq s_1,\\ \lambda_Q(n_i+m-1)+\mu_Q-\lambda_Q & \mbox{if } s_1+1\leq i\leq n\end{array}\right.$$ and if $s_1=n,$ then set $N_i=\lambda_Q(n_i+m-1)+\mu_Q-\lambda_Q+l_i,~\mbox{for } 1\leq i\leq s_1=n.$ 
	
	Thus \eqref{eq:l1.7} can be written as: $$\sum\limits_{i=1}^{n+s_3}\frac{N_i\alpha_i}{1+\alpha_ir}-\sum\limits_{i=1}^{q}\frac{q_i\gamma_i}{1+\gamma_ir}=O\left(r^{\mu_Q+\lambda_Q(N+ n(m-1)-1)+L-1}\right)~\mbox{as $r\rightarrow 0,$}$$ where $0\leq s_3\leq l-1.$ Proceeding in similar fashion as in Case 1, we deduce that $q\geq \mu_Q+m-1.$\\
{\bf  Case 3:} $S_1=\emptyset$ and $S_2\neq\emptyset.$\\
We may assume, without loss of generality, that $S_2=\left\{\beta_1,\beta_2,\ldots,\beta_{s_2}\right\}.$ Then $\beta_i=\gamma_i$ for $1\leq i\leq s_2.$ Take $s_4=l-s_2.$\\
{\bf  Subcase 3.1:} $s_4\geq 1.$\\
 Let $\gamma_{q+i}=\beta_{s_2+i}$ for $1\leq i\leq s_4$ and if $s_2<q,$ then set $$Q_i=\left\{\begin{array}{ccc} q_i-l_i & \mbox{if } 1\leq i\leq s_2,\\  q_i & \mbox{if } s_2+1\leq i\leq q,\\  -l_{s_2-q+i} & \mbox{if } q+1\leq i\leq q+s_4.\end{array}\right.$$ If $s_2=q,$ then set $$Q_i=\left\{\begin{array}{cc} q_i-l_i & \mbox{if } 1\leq i\leq s_2,\\ -l_{s_2-q+i} & \mbox{if } q+1\leq i\leq q+s_4.\end{array}\right.$$
{\bf  Subcase 3.2:} $s_4=0.$\\
 If $s_2<q,$ then set $$Q_i=\left\{\begin{array}{cc}q_i-l_i & \mbox{if } 1\leq i\leq s_2,\\ q_i & \mbox{if } s_2+1\leq i\leq q\end{array}\right.$$ and if $s_2=q,$ then set $Q_i=q_i-l_i,~\mbox{for } 1\leq i\leq s_2=q.$ 
	
	Thus \eqref{eq:l1.7} can be written as: 
	\begin{align*}
& \sum\limits_{i=1}^{n}\frac{\left[\lambda_Q(n_i+m-1)+\mu_Q-\lambda_Q\right]\alpha_i}{1+\alpha_ir}-\sum\limits_{i=1}^{q+s_4}\frac{Q_i\gamma_i}{1+\gamma_ir}\\ &=O\left(r^{\mu_Q+\lambda_Q(N+ n(m-1)-1)+L-1}\right) \mbox{ as } r\rightarrow 0,
\end{align*}
 where $0\leq s_4\leq l-1.$ Proceeding in similar way as in Case 1, we deduce that $q\geq \mu_Q+\lambda_Q(m-1).$\\
{\bf  Case 4.} $S_1\neq\emptyset$ and $S_2\neq\emptyset.$\\
We may assume, without loss of generality, that $S_1=\left\{\beta_1,\beta_2,\ldots,\beta_{s_1}\right\},~S_2=\left\{\gamma_1,\gamma_2,\ldots,\gamma_{s_2}\right\}.$ Then $\beta_i=\alpha_i$ for $1\leq i\leq s_1$ and $\gamma_i=\beta_{s_1+i}$ for $1\leq i\leq s_2.$ Take $s_5=l-s_2-s_1.$\\
{\bf  Subcase 4.1:} $s_5\geq 1.$\\
Let $\alpha_{n+i}=u_{s_1+s_2+i}$ for $1\leq i\leq s_5$ and if $s_1<n,$ then set $$N_i=\left\{\begin{array}{ccc} \lambda_Q(n_i+m-1)+\mu_Q-\lambda_Q + l_i & \mbox{if } 1\leq i\leq s_1,\\  \lambda_Q(n_i+m-1)+\mu_Q-\lambda_Q & \mbox{if } s_1+1\leq i\leq n,\\  l_{s_1+s_2-n+i} & \mbox{if } n+1\leq i\leq n+s_5.\end{array}\right.$$ If $s_1=n,$ then set $$N_i=\left\{\begin{array}{cc} \lambda_Q(n_i+m-1)+\mu_Q-\lambda_Q + l_i & \mbox{if } 1\leq i\leq s_1,\\ l_{s_1+s_2-n+i} & \mbox{if } n+1\leq i\leq n+s_5.\end{array}\right.$$ If $s_2<q,$ then set $$Q_i=\left\{\begin{array}{cc}q_i-l_{s_1+i} & \mbox{if } 1\leq i\leq s_2,\\ q_i & \mbox{if } s_2+1\leq i\leq q\end{array}\right.$$ and if $s_2=q,$ then set $Q_i=q_i-l_{s_1+i},~\mbox{for } 1\leq i\leq s_2.$\\
{\bf  Subcase 4.2:} $s_5=0.$\\
If $s_1<n,$ then set $$N_i=\left\{\begin{array}{cc} \lambda_Q(n_i+m-1)+\mu_Q-\lambda_Q + l_i & \mbox{if } 1\leq i\leq s_1,\\ \lambda_Q(n_i+m-1)+\mu_Q-\lambda_Q & \mbox{if } s_1+1\leq i\leq n.\end{array}\right.$$ If $s_1=n,$ then set $N_i=\lambda_Q(n_i+m-1)+\mu_Q-\lambda_Q + l_i~\mbox{for } 1\leq i\leq s_1.$ 
	
Also, if $s_2<q,$ then set $$Q_i=\left\{\begin{array}{cc}q_i-l_{s_1+i} & \mbox{if } 1\leq i\leq s_2,\\ q_i & \mbox{if } s_2+1\leq i\leq q\end{array}\right.$$ and if $s_2=q,$ then set $Q_i=q_i-l_{s_1+i},~\mbox{for } 1\leq i\leq s_2.$
	
Thus in both subcases, \eqref{eq:l1.7} can be written as $$\sum\limits_{i=1}^{n+s_5}\frac{N_i\alpha_i}{1+\alpha_ir}-\sum\limits_{i=1}^{q}\frac{Q_i\gamma_i}{1+\gamma_ir}=O\left(r^{\mu_Q+\lambda_Q(N+ n(m-1)-1)+L-1}\right)~\mbox{as $r\rightarrow 0,$}$$ where $0\leq s_5\leq l-2.$ Proceeding in similar fashion as in Case 1, we deduce that $q\geq\mu_Q+\lambda_Q(m-1).$ 
\end{proof}

\begin{lem}\label{lem:3}
Let $\left\{f_j\right\}\subset\mathcal{M}(\mD)$ be a sequence of non-vanishing functions, all of whose poles have multiplicities at least $m,~m\in\mN.$ Let $\left\{h_j\right\}\subset\mathcal{H}(\mD)$ be such that $h_j\rightarrow h$ locally uniformly in $\mD,$ where $h\in\cH(\mD)$ and $h(z)\neq 0$ in $\mD.$ If, for each $j,$ $Q[f_j]-h_j$ has at most $\mu_Q + \lambda_Q(m-1)-1$ zeros, ignoring multiplicities, in $\mD,$ then $\left\{f_j\right\}$ is normal in $\mD.$ 
\end{lem}

\begin{proof}
Without loss generality, suppose that $\left\{f_j\right\}$ is not normal at $0\in \mD.$ Then by Lemma \ref{lem:zp}, there exists a sequence of points $\left\{z_j\right\}\subset\mD$ with $z_j\rightarrow 0,$ a sequence of positive real numbers satisfying $\rho_j\rightarrow 0,$ and a subsequence of $\left\{f_j\right\},$ again denoted by $\left\{f_j\right\}$ such that the sequence $$F_j(\zeta):=\frac{f_j(z_j+\rho_j\zeta)}{\rho_j^{(\mu_Q-\lambda_Q)/\lambda_Q}}\rightarrow F(\zeta),$$ spherically locally uniformly in $\mathbb{C},$ where $F\in\cM(\mC)$ is a non-constant and non-vanishing function having poles of multiplicity at least $m.$ Clearly, $Q[F_j]\rightarrow Q[F]$ spherically uniformly in every compact subset of $\mC$ disjoint from poles of $F.$ Also, one can easily see that $Q[F_j](\zeta)=Q[f_j](z_j+\rho_j\zeta).$ Thus, for every $\zeta\in\mC\setminus\left\{F^{-1}(\infty)\right\},$ $$Q[f_j](z_j+\rho_j\zeta)-h_j(z_j+\rho_j\zeta)=Q[F_j](\zeta)-h_j(z_j+\rho_j\zeta)\rightarrow Q[F](\zeta)-h(0)$$ spherically locally uniformly. Since $F$ is non-constant and $x_0>0,~x_i\geq y_i$ for all $i=1,~2,~\ldots,~k,$ by a result of Grahl \cite[Theorem 7]{grahl}, it follows that $Q[F]$ is non-constant. Next, we claim that $Q[F]-h(0)$ has at most $\mu_Q + \lambda_Q(m-1)-1$ zeros in $\mC.$ 

Suppose, on the contrary, that $Q[F]-h(0)$ has $\mu_Q + \lambda_Q(m-1)$ distinct zeros in $\mC,$ say $\zeta_i,~i=1,~2,~\ldots,~\mu_Q + \lambda_Q(m-1).$ Then by Hurwitz's theorem, there exit sequences $\zeta_{j, i},~i=1,~2,~\ldots,~\mu_Q + \lambda_Q(m-1)$ with $\zeta_{j, i}\rightarrow\zeta_i$ such that for sufficiently large $j,$ $Q[f_j](z_j+\rho_j\zeta_{j,i})-h_j(z_j+\rho_j\zeta_{j,i})=0$ for $i=1,~2,~\ldots,~\mu_Q + \lambda_Q(m-1).$ However, $Q[f_j]-h_j$ has at most $\mu_Q + \lambda_Q(m-1)-1$ distinct zeros in $\mD.$ This proves the claim. Now, from Corollary \ref{cor:1}, it follows that $F$ must be a rational function which contradicts Proposition \ref{prop:2}.
\end{proof}

\begin{prop}\label{prop:3}
Let $t$ be a positive integer. Let $\left\{f_j\right\}\subset\mathcal{M}(\mD)$ be a sequence of non-vanishing functions, all of whose poles have multiplicities at least $m,~m\in\mN$ and let $\left\{h_j\right\}\subset\mathcal{H}(\mD)$ be such that $h_j\rightarrow h$ locally uniformly in $\mD,$ where $h\in\cH(\mD)$ and $h(z)\neq 0.$ If, for every j, $Q[f_j](z)-z^th_j(z)$ has at most $\mu_Q + \lambda_Q(m-1)-1$ zeros in $\mD,$ then $\left\{f_j\right\}$ is normal in $\mD.$ 
\end{prop}

\begin{proof}
In view of Lemma \ref{lem:3}, it suffices to prove that $\mathcal{F}$ is normal at $z=0.$ Since $h(z)\neq 0$ in $\mD,$ it can be assumed that $h(0)=1.$ Now, suppose, on the contrary, that $\left\{f_j\right\}$ is not normal at $z=0.$ Then by Lemma \ref{lem:zp}, there exists a subsequence of $\left\{f_j\right\},$ which, for simplicity, is again denoted by $\left\{f_j\right\},$ a sequence of points $\left\{z_j\right\}\subset \mD$ with $z_j\rightarrow 0$ and a sequence of positive real numbers satisfying $\rho_j\rightarrow 0$ such that the sequence $$F_j(\zeta):=\frac{f_j(z_j+\rho_j\zeta)}{\rho_j^{(t+\mu_Q-\lambda_Q)/\lambda_Q}}\rightarrow F(\zeta)$$ spherically locally uniformly in $\mathbb{C},$ where $F\in\mathcal{M}(\mC)$ is a non-constant function. Also, since each $f_j$ is non-vanishing, it follows that $F$ is non-vanishing.
We now distinguish two cases.\\
{\bf Case 1:} Suppose that there exists a subsequence of $z_j/\rho_j,$ again denoted by $z_j/\rho_j,$ such that $z_j/\rho_j\rightarrow\infty$ as $j\rightarrow\infty.$

Define $$g_j(\zeta):=z_j^{-(t+\mu_Q-\lambda_Q)/\lambda_Q}f_j(z_j+z_j\zeta).$$ Then an elementary computation shows that 
	$$Q[g_j](\zeta)= z_j^{-t}Q[f_j](z_j+z_j\zeta),$$
and hence 
\begin{align*}Q[f_j](z_j+z_j\zeta)&-(z_j+z_j\zeta)^th_j(z_j+z_j\zeta)\\
&=z_j^tQ[g_j](\zeta)-(z_j+z_j\zeta)^th_j(z_j+z_j\zeta)\\
&=z_j^t\left[Q[g_j](\zeta)-(1+\zeta)^th_j(z_j+z_j\zeta)\right].
\end{align*}

Since $(1+\zeta)^th_j(z_{j}+z_{j}\zeta)\rightarrow(1+\zeta)^t\neq 0$ in $\mD,$ and $Q[f_j](z_j+z_j\zeta)-z_j^t(1+\zeta)^th_j(z_j+z_j\zeta)$ has at most $\mu_Q + \lambda_Q(m-1)-1$ zeros in $\mD,$ by Lemma \ref{lem:3}, it follows that $\left\{g_j\right\}$ is normal in $\mD$ and so there exists a subsequence of $\left\{g_j\right\},$ again denoted by $\left\{g_j\right\},$ such that $g_j\rightarrow g$ spherically locally uniformly in $\mD,$ where $g\in\cM(\mD)$ or $g\equiv \infty.$  If $g\equiv\infty,$ then 
\begin{align*}
 F_j(\zeta)&=\rho_{j}^{-(t+\mu_Q-\lambda_Q)/\lambda_Q}f_j(z_j+\rho_j\zeta)\\
&=\left(\frac{z_j}{\rho_j}\right)^{(t+\mu_Q-\lambda_Q)/\lambda_Q}z_j^{-(t+\mu_Q-\lambda_Q)/\lambda_Q}f_j(z_j+\rho_j\zeta)\\
&=\left(\frac{z_j}{\rho_j}\right)^{(t+\mu_Q-\lambda_Q)/\lambda_Q}g_j\left(\frac{\rho_j}{z_j}\zeta\right) 
\end{align*}
converges spherically locally uniformly to $\infty$ in $\mC$ showing that $F\equiv\infty,$ a contradiction to the fact that $F$ is non-constant. Since $g_j(\zeta)\neq 0,$ it follows that either $g(\zeta)\neq 0$ or $g\equiv 0.$ If $g(\zeta)\neq 0,$ then by the previous argument, we find that $F\equiv\infty,$ a contradiction. If $g\equiv 0,$ then choose $n\in\mN$ such that $n+1>(t+\mu_Q-\lambda_Q)/(\lambda_Q).$ Thus, for each $\zeta\in\mC,$ we have 
\begin{align*}
F_j^{(n+1)}(\zeta)&=\rho_{j}^{-(t+\mu_Q-\lambda_Q)/\lambda_Q+n+1}f_j^{(n+1)}(z_j+\rho_j\zeta)\\
&=\left(\frac{\rho_j}{z_j}\right)^{-(t+\mu_Q-\lambda_Q)/\lambda_Q+n+1}g_j^{(n+1)}\left(\frac{\rho_j}{z_j}\zeta\right).
\end{align*}
Therefore, $F_j^{(n+1)}(\zeta)\rightarrow 0$ spherically uniformly which implies that $F$ is a polynomial of degree at most $n,$ a contradiction to the fact that $F$ is non-constant and non-vanishing meromorphic function.\\
{\bf Case 2:} Suppose that there exists a subsequence of $z_j/\rho_j,$ again denoted by $z_j/\rho_j,$ such that $z_j/\rho_j\rightarrow\alpha$ as $j\rightarrow\infty,$ where $\alpha\in\mathbb{C}.$ Then $$G_j(\zeta)=\rho_j^{-(t+\mu_Q-\lambda_Q)/\lambda_Q}f_j(\rho_j\zeta)=F_j\left(\zeta-\frac{z_j}{\rho_j}\right)\rightarrow F(\zeta-\alpha):= G(\zeta),$$ spherically locally uniformly in $\mC.$ Clearly, $G(\zeta)\neq 0.$ Also, it is easy to see that $Q[G_j](\zeta)=\rho_j^{-t}Q[f_j](\rho_j\zeta).$ Thus, $$Q[G_j](\zeta)-\zeta^t h_j(\rho_j\zeta)=\frac{Q[f_j](\rho_j\zeta)-(\rho_j\zeta)^t h_j(\rho_j\zeta)}{\rho_j^t}\rightarrow Q[G](\zeta)-\zeta^t$$ spherically uniformly in every compact subset of $\mathbb{C}$ disjoint from the poles of $G.$ Clearly, $Q[G](\zeta)\not\equiv\zeta^t,$ otherwise $G$ has to be a polynomial, which is not possible since $G(\zeta)\neq 0.$ Since $Q[f_j](\rho_j\zeta)-(\rho_j\zeta)^t h_j(\rho_j\zeta)$ has at most $\mu_Q + \lambda_Q(m-1)-1$ distinct zeros in $\mD,$ it follows that $Q[G](\zeta)-\zeta^m$ has at most $\mu_Q + \lambda_Q(m-1)-1$ distinct zeros in $\mC$ and hence by Corollary \ref{cor:1}, $G$ must be a rational function. However, this contradicts Proposition \ref{prop:2}. Hence $\cF$ is normal in $\mD.$
\end{proof}

\section{Proof of Theorem \ref{thm:1}}
We use ideas from Deng et al. \cite{deng-1} to prove this result. By virtue of Lemma \ref{lem:3}, it is sufficient to prove the normality of $\left\{f_j\right\}$ at points $z\in\mD$ where $h(z)=0.$ Without loss of generality, we assume that $h(z)=z^t a(z),$ where $t\in\mN,~a\in\cH(\mD),~a(z)\neq 0$ and $a(0)=1.$ Further, since $h_j\rightarrow h$ locally uniformly in $\mD,$ we can assume that $$h_j(z)=(z-z_{j,1})^{t_1}(z-z_{j,2})^{t_2}\cdots (z-z_{j,l})^{t_l} a_j(z),$$ where $\sum\limits_{i=1}^{l}t_i=t,$ $z_{j,i}\rightarrow 0$ for $1\leq i\leq l$ and $a_j(z)\rightarrow a(z)$ locally uniformly in $\mD.$ Again, we may assume that $z_{j,1}=0,$ since $\left\{f_j(z)\right\}$ is normal in $\mD$ if and only if $\left\{f_j(z+z_{j,1})\right\}$ is normal in $\mD$ (see \cite[p. 35]{schiff}). Now, we shall prove the normality of $\left\{f_j\right\}$ at $z=0$ by applying the principle of mathematical induction on $t.$\\
Note that if $t=1,$ then $l=1$ and so $h_j(z)=za_j(z).$ Thus, by Proposition \ref{prop:3}, $\left\{f_j\right\}$ is normal at $z=0.$ Also, if $l=1,$ then $h_j(z)=z^ta_j(z),$ and again by Proposition \ref{prop:3}, $\left\{f_j\right\}$ is normal at $z=0.$ So, let $l\geq 2$ and for $n\in\mN$ with $1<t<n,$ suppose that $\left\{f_j\right\}$ is normal at $z=0.$ In accordance with the principle of mathematical induction, we only need to show that $\left\{f_j\right\}$ is normal at $z=0$ when $n=t.$\\
By rearranging the zeros of $h_j,$ if necessary, we can assume that $|z_{j,i}|\leq |z_{j,l}|$ for $2\leq i\leq l.$ Let $z_{j,l}=w_j.$ Then $w_j\rightarrow 0.$ Define $$g_j(z):=\frac{f_j(w_jz)}{w_j^{(t+\mu_Q-\lambda_Q)/\lambda_Q}} \mbox{ and } v_j(z):=\frac{h_j(w_jz)}{w_j^t},~z\in D_{r_j}(0),~r_j\rightarrow\infty.$$ Then an easy computation shows that $Q[g_j](z)=w_j^{-t}Q[f_j](w_jz)$ and $$v_j(z)=z^{t_1}\left(z-\frac{z_{j,2}}{w_j}\right)^{t_2}\cdots\left(z-\frac{z_{j,l-1}}{w_j}\right)^{t_{l-1}}(z-1)^{t_{l}}a_j(w_jz)\rightarrow v(z)$$ locally uniformly in $\mC.$ Clearly, $0$ and $1$ are two distinct zeros of $v$ and hence all zeros of $v$ have multiplicities at most $t-1.$ Since 
\begin{equation}\label{eq:t1}
Q[g_j](z)-v_j(z)=\frac{Q[f_j](w_jz)-h_j(w_jz)}{w_j^{t}}
\end{equation}
 and $Q[f_j](w_jz)-h_j(w_jz)$ has at most $\mu_Q + \lambda_Q(m-1)-1$ distinct zeros, it follows that $Q[g_j](z)-v_j(z)$ has at most $\mu_Q + \lambda_Q(m-1)-1$ distinct zeros in $\mC.$ Thus, by induction hypothesis, we find that $\left\{g_j\right\}$ is normal in $\mC.$ Suppose that $g_j\rightarrow g$ spherically locally uniformly in $\mC.$ Then either $g\in\cM(\mC)$ or $g\equiv\infty.$\\
{\bf Case 1:} $g\in\cM(\mC).$\\
Since $g_j(z)\neq 0,$ it follows that either $g(z)\neq 0$ or $g\equiv 0.$ First, suppose that $g(z)\neq 0.$ Since $g_j\rightarrow g$ spherically locally uniformly in $\mC,$ it follows that $Q[g_j]\rightarrow Q[g]$ in every compact subset of $\mC$ disjoint from the poles of $g.$ Then from \eqref{eq:t1}, we find that $Q[g]-v$ has at most $\mu_Q + \lambda_Q(m-1)-1$ distinct zeros in $\mC$ and thus by Corollary \ref{cor:1} and Proposition \ref{prop:2}, $g$ has to be a constant.\\

Next, we claim that $\left\{f_j\right\}$ is holomorphic in $D_{\delta/2}(0)$ for some $\delta\in (0,1).$ Suppose, on the contrary, that $\left\{f_j\right\}$ is not holomorphic in $D_{\delta/2}(0)$ for any $\delta\in (0,1).$ Then there exists sequence $\eta_j\in D_{\delta/2}(0)$ such that $\eta_j\rightarrow 0$ and $f_j(\eta_j)=\infty.$ Assume that $\eta_j$ has the smallest modulus among the poles of $f_j.$ It is easy to see that $\eta_j/w_j\rightarrow\infty,$ otherwise, $$f_j(\eta_j)=w_j^{(t+\mu_Q-\lambda_Q)/\lambda_Q}g_j(\eta_j/w_j)\rightarrow 0, \mbox{ a contradiction }.$$ Let $$\psi_j(z):=\frac{f_j(\eta_jz)}{\eta_j^{(t+\mu_Q-\lambda_Q)/\lambda_Q}} \mbox{ and } u_j(z):= \frac{h_j(\eta_jz)}{\eta_j^{t}},~z\in D_{r_j}(0),~r_j\rightarrow\infty.$$ Then 
\begin{equation}\label{eq:t2}
Q[\psi_j](z)-u_j(z)=\frac{Q[f_j](\eta_jz)-h_j(\eta_jz)}{\eta_j^{t}}
\end{equation} 
and $$u_j(z)=z^{t_1}\left(z-\frac{z_{j,2}}{\eta_j}\right)^{t_2}\cdots\left(z-\frac{w_j}{\eta_j}\right)^{t_{l}}a_j(\eta_jz)\rightarrow z^t$$ locally uniformly in $\mC.$ From Lemma \ref{lem:3}, it follows that $\left\{\psi_j\right\}$ is normal in $\mC\setminus\left\{0\right\}.$ Since $\psi_j(z)\neq 0$ and $\psi_j$ is holomorphic in $\mD,$ one can easily see that $\left\{\psi_j\right\}$ is normal in $\mD$ and hence in $\mC.$ Assume that $\psi_j\rightarrow\psi$ spherically locally uniformly in $\mC,$ where $\psi\in\cM(\mC)$ or $\psi\equiv\infty.$ Since $$\psi_j(0)=\frac{f_j(0)}{\eta_j^{(t+\mu_Q-\lambda_Q)/\lambda_Q}}=\left(\frac{w_j}{\eta_j}\right)^{(t+\mu_Q-\lambda_Q)/\lambda_Q}g_j(0)\rightarrow 0,$$ therefore, $\psi\not\equiv\infty.$ Also, since $\psi_j(z)\neq 0,$ we have $\psi(z)\neq 0$ or $\psi\equiv 0.$ However, the latter is not possible since $\infty=\psi_j(1)\rightarrow\psi(1)=\infty.$ Thus, $\psi(z)\neq 0.$ Note that $Q[\psi_j](z)-u_j(z)\rightarrow Q[\psi](z)-z^t$ spherically uniformly in every compact subset of $\mC$ disjoint from the poles of $\psi,$ so by \eqref{eq:t2}, we conclude that $Q[\psi](z)-z^t$ has at most $\mu_Q + \lambda_Q(m-1)-1$ distinct zeros in $\mC.$ By Corollary \ref{cor:1} and Proposition \ref{prop:2}, $\psi$ reduces to a constant, which contradicts the fact that $\infty=\psi_j(1)\rightarrow\psi(1)=\infty.$ Hence $\left\{f_j\right\}$ is holomorphic in $D_{\delta/2}(0).$ Since $f_j(z)\neq 0,$ it follows that $\left\{f_j\right\}$ is normal   at $z=0.$\\
 Next, suppose that $g\equiv 0.$ Then by the preceding discussion, one can easily see that $\left\{f_j\right\}$ is holomorphic in $D_{\delta/2}(0)$ and hence $\left\{f_j\right\}$ is normal at $z=0.$\\ 
{\bf Case 2:} $g\equiv\infty.$\\
Let $\phi_j(z):=f_j(z)/z^{(t+\mu_Q-\lambda_Q)/\lambda_Q}.$ Then $1/\phi_j(0)=0.$\\
{\bf Subcase 2.1:} When $\left\{1/\phi_j\right\}$ is normal at $z=0.$\\
Then $\left\{\phi_j\right\}$ is normal at $z=0$ and so there exists $r>0$ with $D_{r}(0)\subseteq\mD$ such that $\left\{\phi_j\right\}$ is normal in $D_{r}(0).$ Assume that $\phi_j\rightarrow\phi$ spherically locally uniformly. Since $\phi_j(0)=\infty,$ there exists $\rho>0$ such that, for sufficiently large $j,$ $|\phi_j(z)|\geq 1$ in $D_{\rho}(0)\subset D_{r}(0).$ Also, since $f_j(z)\neq 0$ in $D_{\rho}(0),$ $1/f_j$ is holomorphic in $D_\rho(0)$ and hence $$\left|\frac{1}{f_j(z)}\right|=\left|\frac{1}{\phi_j(z)}\cdot\frac{1}{z^{(t+\mu_Q-\lambda_Q)/\lambda_Q}}\right|\leq\left(\frac{2}{\rho}\right)^{(t+\mu_Q-\lambda_Q)/\lambda_Q} \mbox{ in } \partial D_{\rho/2}(0).$$ Then by the maximum principle and Montel's theorem \cite[p.35]{schiff}, we conclude that $\left\{f_j\right\}$ is normal at $z=0.$\\
{\bf Subcase 2.2:} When $\left\{1/\phi_j\right\}$ is not normal at $z=0.$\\
 By Montel's theorem, it follows that, for every $\epsilon>0,$ $\left\{1/\phi_j(z)\right\}$ is not locally uniformly bounded in $D_{\epsilon}(0).$ Therefore, we can find a sequence $\epsilon_j\rightarrow 0$ such that $1/\phi_j(\epsilon_j)\rightarrow\infty.$ Since $\left|1/\phi_j\right|$ is continuous, there exists $b_j\rightarrow$ such that $\left|1/\phi_j(b_j)\right|=1.$\\
 Let $$K_j(z):=\frac{f_j(b_jz)}{b_j^{(t+\mu_Q-\lambda_Q)/\lambda_Q}},~z\in D_{r_j}(0),~r_j\rightarrow\infty \mbox{ and } q_j(z):=\frac{h_j(b_jz)}{b_j^t}.$$ Then $K_j(z)\neq 0$ and a simple computation shows that $$Q[K_j](z)=\frac{Q[f_j](b_jz)}{b_j^{t}} \mbox{ and } q_j(z)=z^{t_1}\left(z-\frac{z_{j,2}}{b_j}\right)^{t_2}\cdots\left(z-\frac{w_{j}}{b_j}\right)^{t_l}a_j(b_jz).$$ Note that $$g_j\left(\frac{b_j}{w_j}\right)=\frac{f_j(b_j)}{w_j^{(t+\mu_Q-\lambda_Q)/\lambda_Q}}=\frac{f_j(b_j)}{b_j^{(t+\mu_Q-\lambda_Q)/\lambda_Q}}\cdot\left(\frac{b_j}{w_j}\right)^{(t+\mu_Q-\lambda_Q)/\lambda_Q}\rightarrow\infty.$$ Since $|1/\phi_j(b_j)|=1$ and $(t+\mu_Q-\lambda_Q)/\lambda_Q>0,$ it follows that $b_j/w_j\rightarrow\infty,$ and hence $w_j/b_j\rightarrow 0.$ This implies that $q_j(z)\rightarrow z^t$ locally uniformly in $\mC.$ Further, since $Q[K_j](z)-q_j(z)=(Q[f_j](b_jz)-h_j(b_jz))/b_j^t,$ it follows that $Q[K_j]-q_j$ has at most $\mu_Q + \lambda_Q(m-1)-1$ distinct zeros in $\mC$ and hence by Lemma \ref{lem:3}, $\left\{K_j\right\}$ is normal in $\mC\setminus\left\{0\right\}.$ We claim that $\left\{K_j\right\}$ is normal in $\mC.$ Suppose otherwise. Then by Lemma \ref{lem:xu}, there is a subsequence of $\left\{K_j\right\}$, which for the sake of convenience, is again denoted by $\left\{K_j\right\}$ such that $K_j(z)\rightarrow 0$ in $\mC\setminus\left\{0\right\},$ which is not possible since $|K_j(1)|=1.$ This establishes the claim. Now, suppose that $K_j\rightarrow K$ spherically locally uniformly in $\mC.$ It is evident that $K(z)\neq 0$ in $\mC$ and $K\not\equiv\infty,$ as $K(1)=1.$ Then $Q[K_j]\rightarrow Q[K]$ spherically  uniformly in every compact subset of $\mC$ disjoint from the poles of $K.$ Since $Q[K_j]-q_j$ has at most $\mu_Q + \lambda_Q(m-1)-1$ distinct zeros in $\mC,$ it follows that $Q[K]-z^t$ has at most $\mu_Q + \lambda_Q(m-1)-1$ distinct zeros in $\mC$ and so by Corollary \ref{cor:1} and Proposition \ref{prop:2}, $K$ reduces to a constant. Using the same arguments as in Case 1, we find that $\left\{f_j\right\}$ is normal at $z=0.$ This completes the induction process and hence the proof.

\begin{flushleft}
{\bf Acknowledgment.} The author is thankful to Prof. N. V. Thin for a careful reading of the manuscript and pointing out references \cite{thin-1}, \cite{thin-2} and \cite{thin-3}.

\end{flushleft}




\bibliographystyle{amsplain}

\end{document}